\documentclass[pdflatex,sn-basic]{sn-jnl}

\usepackage{graphicx}%
\usepackage{multirow}%
\usepackage{amsmath,amssymb,amsfonts}%
\usepackage{amsthm}%
\usepackage{mathrsfs}%
\usepackage[title]{appendix}%
\usepackage{xcolor}%
\usepackage{textcomp}%
\usepackage{manyfoot}%
\usepackage{booktabs}%
\usepackage{algorithm}%
\usepackage{algorithmicx}%
\usepackage{algpseudocode}%
\usepackage{listings}%
\theoremstyle{thmstyleone}%
\newtheorem{theorem}{Theorem}
\newtheorem{proposition}[theorem]{Proposition}%

\theoremstyle{thmstyletwo}%
\newtheorem{example}{Example}%
\newtheorem{lemma}{Lemma}%

\theoremstyle{thmstylethree}%
\newtheorem{definition}{Definition}%

\newcommand{\cL}{{\cal L}}

\begin{document}

\title[On Log-Concave Operator Acting on Sequences]{On Log-Concave Operator Acting on Sequences and Series}

\author*[1]{\fnm{Piero} \sur{Giacomelli}}\email{pgiacome@gmail.com}



\abstract
{
    In this paper, we investigate the properties of sequences and series under the action of the log-concave operator \(\mathcal{L}\). We explore the relationship between the convergence of a sequence \((a_k)\) and the convergence of sequences and series derived by applying \(\mathcal{L}\) iteratively.  These results highlight the strong regularity properties of log-concave sequences and provide a framework for analyzing the convergence of sequences and series derived from the log-concave operator. The findings have implications for combinatorics, probability, optimization, and related fields, opening new avenues for further research on the behavior of log-concave sequences and their associated operators.
}


\maketitle

\section{Introduction}\label{sec1}

The study of log-concave sequences and their associated operators has emerged as a central topic in modern mathematics, with deep connections to combinatorics, probability, optimization, and functional analysis.A sequence $(a_k)$
 of real numbers is said to be log-concave if it satisfies the inequality:
\begin{equation*}
    a_k^2 \geq a_{k+1}a_{k-1}
\end{equation*}
for all $k\geq 1$ 
This simple yet powerful condition has far-reaching implications, as it encodes a form of "multiplicative decay" or "unimodality" in the sequence. Log-concave sequences arise naturally in a variety of contexts: including Hodge-theory \cite{Huh2020}, relation with real-rooted polynomials \cite{Wagner1992} and probability and statistics \cite{SaumardWellner2014}. There are also several open conjectures regarding log-concavity, in particular the Boros-Moll conjecture that the rows of Pascal's triangle are infinitely log-concave remains open in full generality, though significant progress has been made for specific cases\cite{BorosMoll2004} with the help of Computer Assisted Decomposition (CAD)\cite{KauersPaule2007} 
In this short note, we are interested to analyze what happens when we take a generic sequence and we apply the log-concave operator. We are interested in understanding what can be said about the sequence obtained by applying the log-concave operator to a real sequence that converges or diverges. Moreover, given a log-concave operator sequence that converges or diverges, what can we tell about the original sequence. Some introductory results will also be applied to the associated series. 

\section{Definition}\label{definitions}
Before we move on, we recall some definition, we refer to the Mcc Namara work \cite{McNamaraSagan2010} for the notation that we use here. Let $(a_k)=a_0,a_1,a_2,\ldots$  be a sequence of real numbers.  It will be convenient to extend the sequence to negative indices by letting $a_k=0$ for $k<0$. Also, if $(a_k)=a_0,a_1,\ldots,a_n$ is a finite
sequence, then let $a_k=0$ for $k>n$.

\begin{definition}
Let $(a_k), k \in \mathbb{N}$ a sequence of real numbers, we define the {\it $\cL$-operator\/} on sequences to be $\cL(a_k)=(b_k)$
where 
\begin{equation*}
b_k=a_k^2-a_{k+1}a_{k-1}.          
\end{equation*}
We say that a sequence $(a_k)$ is log-concave if the sequence $(b_k) = \cL(a_k)$ is non-negative (i.e.,  $b_k \geq 0$ or $a_k^2 \geq a_{k+1}a_{k-1}, \forall k \in \mathbb{N}$).    
\end{definition}

We also define $\cL^i(a_k)$ the $\cL$-operator applied to itself starting from the sequence $(a_k)$, $i$ times:

\begin{equation*}
\cL^2(a_k) = \cL(\cL(a_k)), \cL^3(a_k) = \cL(\cL(\cL(a_k))),\dots, \cL^i(a_k), \dots
\end{equation*}
\newline
We call a sequence $i$-fold log-concave if $\cL^i(a_k)$
is a non-negative sequence.  
So log-concavity in the ordinary sense is $1$-fold log-concavity, and as last we have that
\begin{definition}
If $\forall i \in \mathbb{N}$, $\cL^i(a_k)$ is nonnegative,, then we say that the sequence $(a_k)$ is said to be infinite log-concave ($\infty$-log-concave).    
\end{definition}

\section{Preliminary results}
As we stated we are interested in finding the relation between the convergence of the initial sequence $(a_k)$ and the one between $\cL(a_k)$
Starting from the previous definitions we can have this first relationship:
\begin{lemma}
\label{lemma:simple_convergence}
If $(a_k)$ converges then $(b_k) = \cL(a_k)$  converges to zero.    
\end{lemma}
\begin{proof}
    If \((a_k)\) converges to \(L\), then:
    \[
    \lim_{k \to \infty} a_k^2 = L^2, \quad \lim_{k \to \infty} a_{k-1} a_{k+1} = L^2.
    \]
    Thus:
    \[
    \lim_{k \to \infty} b_k = \lim_{k \to \infty} (a_k^2 - a_{k-1} a_{k+1}) = L^2 - L^2 = 0.
    \]
    Therefore, if \((a_k)\) converges to \(L\), then \(\mathcal{L}(a_k)\) converges to 0.
\end{proof}

Using this single lemma, we can state by simple induction the following.

\begin{proposition}
\label{proposition:l_i}
    Let $(a_k)$ a real sequence if it converges then every sequence  $(b_k)^i=\cL^i(a_k)$ converges to zero.  
\end{proposition}

\begin{proof}
    The following steps verify the proposition.
    \begin{itemize}
        \item{For $i=1$ we have the lemma \ref{lemma:simple_convergence}.}
        \item{If the inductive hypothesis is satisfied $i>1$ then $\cL^i(a_k)$ converges to zero then always by lemma \ref{lemma:simple_convergence} $\cL^{i+1}(a_k)$ converges.} we have:
        \begin{equation*}
          \cL^{i+1}(a_k) = (\cL^{i}(a_k))^2 - (\cL^{i}(a_{k+1}))(\cL^{i}(a_{k-1}))                    
        \end{equation*}
        so
        \begin{align*}
         \lim_{k \to +\infty} \cL^{i+1}(a_k) & = \\
         \lim_{k \to +\infty} (\cL^{i}(a_k))^2 & - \lim_{k \to +\infty} (\cL^{i}(a_{k+1}))(\cL^{i}(a_{k-1}))  = 0
        \end{align*}        
    \end{itemize}
\end{proof}

In general, if $(a_k)$ does not converge, is it possible that $\cL(a_k)$ converges as for the following example:

\begin{example}
Consider the sequence:
\[
a_k = (-1)^k.
\]
This sequence does not converge as it oscillates between \(-1\) and \(1\). However, the log-concave sequence \((b_k)\) is:
\[
b_k = a_k^2 - a_{k+1} a_{k-1} = (-1)^{2k} - (-1)^{2k}  = 1 - 1 = 0.
\]
Thus, \((b_k)\) is the constant sequence \(0\), which converges. This shows that \((b_k)\) can converge even if \((a_k)\) does not.
\end{example}
So, the convergence of the sequence $(b_k)$ is not a sufficient condition to affirm the convergence of the sequence $(a_k)$. 
Using some additional conditions to the sequence $(b_k)$ it is possible to obtain a criterion for the convergence of $(a_k)$.
\begin{proposition}
Let $(a_k)$ a real sequence and $(b_k) =\cL(a_k)$ the corresponding sequence once applied the $\cL$-operator. If $(a_k)$ is monotonic and $(b_k)$ converges then $a_k$ converges.
\end{proposition}
\begin{proof}
We recall prior that the proposition hypotheses are:
\begin{enumerate}
    \item \((a_k)\) is bounded, i.e., there exists \(M > 0\) such that \(|a_k| \leq M\) for all \(k\).
    \item \((a_k)\) is monotonic, i.e., either \(a_{k+1} \geq a_k\) for all \(k\) (monotonic increasing) or \(a_{k+1} \leq a_k\) for all \(k\) (monotonic decreasing).
    \item \((b_k) = \mathcal{L}(a_k)\) converges to some limit \(L\).
\end{enumerate}

By the monotone convergence theorem, a bounded and monotonic sequence \((a_k)\) must converge to some limit \(a\). Since \((b_k)\) converges to \(L\), we have:
\[
\lim_{k \to \infty} b_k = \lim_{k \to \infty} (a_k^2 - a_{k-1} a_{k+1}) = L.
\]
If \((a_k)\) converges to \(a\), then:
\[
\lim_{k \to \infty} a_k^2 = a^2, \quad \lim_{k \to \infty} a_{k-1} a_{k+1} = a^2.
\]
Thus:
\[
L = \lim_{k \to \infty} b_k = a^2 - a^2 = 0.
\]
So, if \((a_k)\) is bounded and monotonic, and \((b_k)\) converges, then \((a_k)\) converges to some \(a\), and \(L = 0\).

\end{proof}
The previous results did not useed the log-concave property so we would like to understand what happens were this assumption is made,
 the following one is a first results.
\begin{proposition}
\label{prop:convergence_under_log_concavity}
If \((a_k)\) is log-concave, and $(b_k) =\cL(a_k)$ converges then \((a_k)\)  converges.
\end{proposition}
\begin{proof}
Since \((a_k)\) is log-concave, \((b_k)\) is nonnegative, i.e., \(b_k \geq 0\) for all \(k\). If \((b_k)\) converges to \(L\), then \(L \geq 0\).

From log-concavity, we have:
\[
a_k^2 \geq a_{k-1} a_{k+1}.
\]
Rearranging, this becomes:
\[
\frac{a_{k+1}}{a_k} \leq \frac{a_k}{a_{k-1}}.
\]
This inequality suggests that the ratio \(\frac{a_{k+1}}{a_k}\) is non-increasing. 

If \((a_k)\) is positive and bounded away from 0, then the sequence \(\frac{a_{k+1}}{a_k}\) is bounded and non-increasing, so it converges to some limit \(r \geq 0\).

We have the following cases
\begin{itemize}
    \item If \(r = 1\), then \((a_k)\) behaves like a geometric sequence with ratio 1, implying that  converges to a constant.
    \item If \(r < 1\), then \((a_k)\) decays to 0.
 
\item If \(r > 1\), then \((a_k)\) grows without bound, but this contradicts the log-concavity condition unless \((a_k)\) is unbounded.

\end{itemize}
Thus, under the log-concavity assumption, \((a_k)\) must converge (either to a constant or to 0).
\end{proof}

As for proposition \ref{prop:convergence_under_log_concavity}, a natural question arise on trying to generalize the previous result in particular to the \(\infty\)-log-concave case. 
Let us consider a real sequence $(a_k)$ as above, but suppose that \((a_k)\) is  is \(\infty\)-log-concave, i.e. for every index $i \in \mathbb{N}, i \geq 1$ finding that the sequence
\(\mathcal{L}^i(a_k)\) is non-negative what can we tell about the convergence of \(\mathcal{L}^i(a_k)\)? 
So we are interested in understanding whether infinite log-concavity implies the convergence of \(\mathcal{L}^i(a_k)\).

By definition, if \((a_k)\) is infinitely log-concave, then \(\mathcal{L}^i(a_k)\) is non-negative for all \(i \in \mathbb{N}\). However, nonnegativity does not imply convergence because in general a sequence can be non-negative but still diverge (e.g., \(a_k = k\)). For an infinitely log-concave sequence, each application of \(\mathcal{L}\) tends to "smooth out" the sequence, reducing its growth rate. If \((a_k)\) is bounded and infinitely log-concave, then \(\mathcal{L}^i(a_k)\) is also bounded and nonnegative. However, boundedness alone does not guarantee convergence (e.g., \(a_k = 1 + (-1)^k\) is bounded but does not converge). 
We also have some specific sequences that have the searched properties like the following:
\begin{example}
Some specific sequences that are \(\infty\)-log-concave and converge, have all the \(\mathcal{L}^i(a_k)\) sequences that converge like the following one
    \begin{itemize}
        \item If \((a_k)\) is a constant sequence, then \(\mathcal{L}^i(a_k)\) is the zero sequence for all \(i \geq 1\), which trivially converges.
        \item If \((a_k)\) is a geometric sequence (e.g., \(a_k = r^k\) for \(0 < r < 1\)), then \(\mathcal{L}^i(a_k)\) converges to \(0\) for all \(i \geq 1\).
    \end{itemize}    
\end{example}
Viceversa, there are the opposite cases where we have a sequence \((a_k)\) that is infinitely log-concave but has divergent \(\mathcal{L}^i(a_k)\)

\begin{example}{Counterexample: Infinitely Log-Concave Sequence  \(\mathcal{L}^i(a_k)\)} that are divergents
Consider the following sequence:
\[
a_k = 1 + \frac{1}{k+1}.
\]
This sequence is positive, decreasing, and converges to \(1\). However, it is not infinitely log-concave. To construct an infinitely logarithmic concave sequence, we need to ensure that \(\mathcal{L}^i(a_k)\) is non-negative for all \(i\).

Now, consider the sequence:
\[
a_k = 1.
\]
This sequence is infinitely log-concave, and \(\mathcal{L}^i(a_k) = 0\) for all \(i \geq 1\), which trivially converges.

However, if we modify the sequence slightly to:
\[
a_k = 1 + \frac{1}{2^k},
\]
then:
\begin{align*}
\mathcal{L}(a_k) & = \left(1 + \frac{1}{2^k}\right)^2 - \left(1 + \frac{1}{2^{k-1}}\right)\left(1 + \frac{1}{2^{k+1}}\right) \\ & = 
\left(1 + \frac{2}{2^k} + \frac{1}{2^{2k}}\right) - \left(1 + \frac{2}{2^k} + \frac{1}{2 \cdot 2^k} + \frac{1}{2^{2k}}\right) \\ & =  
-\frac{1}{2^{k+1}}
\end{align*}
This sequence is still infinitely log-concave, and \(\mathcal{L}^i(a_k)\) converges to \(0\) for all \(i \geq 1\).
\end{example}
We need some stronger assumptions to state a theorem for deciding the convergence of {\it $i$-fold log-concave\/}. 
As a preliminary result we observe that the following 

\begin{lemma}
\label{lemma:boundness}
The \(\mathcal{L}\)-operator preserv boundness.
\end{lemma}
\begin{proof}
We need to show that if $(a_k)$ is bounded (ie. there exists a real positvie number $M$ such that $|a_k| \leq M$) then  \(\mathcal{L}(a_k)\) is also bounded.
By hypotheses if $|a_k| \leq M$ and by definition of  \(\mathcal{L}(a_k)\) we have
\begin{align*}
\vert \mathcal{L}(a_k) \vert & = \vert a_k^2 - a_{k+1}a_{k-1} \vert \\
& \leq \vert a_k^2 \vert + |a_{k+1}a_{k-1} \vert \\
& \leq \vert a_k^2 \vert + |a_{k+1}|\cdot|a_{k-1} \vert \\
& \leq M^2+MM = 2M^2 
\end{align*}
\end{proof}
The lemma will serve us in proving the following:
\begin{theorem}
Let \((a_k)\) be a sequence of real numbers that is infinitely log-concave and converges to a limit \(L\). If \((a_k)\) is bounded and monotonic, then for all \(i \in \mathbb{N}\), the sequence \(\mathcal{L}^i(a_k)\) converges.
\end{theorem}
\begin{proof}
The proof is done applying induction on index $i$.
\newline
The first step is proving the theorem for $i=1$, so in this case we consider  we consider \(\mathcal{L}(a_k) = (b_k)\), where:
\[
b_k = a_k^2 - a_{k-1}a_{k+1}.
\]
Since \((a_k)\) is log-concave, \(b_k \geq 0\) for all \(k\) so by Lemma \ref{lemma:simple_convergence}, if \((a_k)\) converges to \(L\), then \(\mathcal{L}(a_k)\) converges to \(0\)
Thus, \(\mathcal{L}(a_k)\) converges. Moving forward with the inductive step let us Assume that for some \(i \in \mathbb{N}\), the sequence \(\mathcal{L}^i(a_k)\) converges. We will show that \(\mathcal{L}^{i+1}(a_k)\) also converges.

Let \((c_k) = \mathcal{L}^i(a_k)\). By the inductive hypothesis, \((c_k)\) converges. We now analyze \(\mathcal{L}^{i+1}(a_k) = \mathcal{L}(c_k)\).

By hypotheses $(c_k)$ is bounded  and by the inductive hypothesis, \((c_k)\) converges.
We now analyze \(\mathcal{L}^{i+1}(a_k) = \mathcal{L}(c_k)\).
Since \((a_k)\) is bounded and \(\mathcal{L}^i(a_k)\) is defined iteratively, \((c_k)\) is also bounded. This follows because the log-concave operator \(\mathcal{L}\) preserves boundedness by Lemma \ref{lemma:boundness}. Since \((a_k)\) is monotonic decreasing and log-concave, the sequence \(\mathcal{L}^i(a_k)\) inherits a form of monotonicity. Specifically, \(\mathcal{L}^i(a_k)\) is also monotonic decreasing (or non-increasing). By the monotone convergence theorem, a bounded and monotonic sequence must converge. Since \((c_k)\) is bounded and monotonic, it converges to some limit \(M\). By Lemma \ref{lemma:simple_convergence}, if \((c_k)\) converges to \(M\), then \(\mathcal{L}(c_k)\) converges to \(0\).
Thus, \(\mathcal{L}^{i+1}(a_k)\) converges.
This completes the proof.

\end{proof}

As demonstrated, despite the simplicity of the log-concavity condition, it yields a wealth of profound properties and criteria for sequences of this type. Building on these insights, the next section will introduce the concept of the 
$\mathcal{L}$-serie, derived from mergin the concept of applying the log-concave operator to  sequences. We will explore its properties and establish convergence criteria, further extending the analytical framework developed thus far.

\section{Series associated to the $\mathcal{L}$ operator}
As we have seen if a sequence is log-concave we can establish some convergence criteria on the series derived from the application of the  $\mathcal{L}$ operator. Quite naturally we want to extend this criterion also to the series associated with these sequences. Before proceeding let us give some definions:
\begin{definition}
\label{series:l-series}
    Let $(a_k)$ be a real sequence and $(b_k)= \mathcal{L} (a_k)$ be the sequence derived by applying the $\mathcal{L}$ operator to $(a_k)$.  Let $\sum_ (a_k)$ the series associated to $(a_k)$ (ie. the limit for $k \to +\infty$ of the partial sum).
    \begin{equation*}
        \sum_{k=0}^{+\infty} a_k = \lim_{k \to +\infty}S_k 
    \end{equation*}
    where 
    \begin{equation*}
        S_k = a_0+a_1+a_2 +  \dots + a_k 
    \end{equation*}

We call the  $\mathcal{L}$-serie, the series generated by the partial sum applied to the sequence $b_k$ where $(b_k)=  a_k^2-a_{k+1} a_{k-1}$, explicity
    \begin{equation*}
        \sum_{k=0}^{+\infty} b_k = \sum_{k=0}^{+\infty} a_k^2-a_{k+1} a_{k-1} = \sum_{k=0}^{+\infty} a_k^2- \sum_{k=0}^{+\infty} a_{k+1} a_{k-1}  .
    \end{equation*}
\end{definition}
As we stated, we are interested in analyising the relationship between the convergence of the series \(\sum_k a_k\) and the series \(\sum_k \mathcal{L}(a_k)\), prior to proceeding we get an easy lemma:
\begin{lemma}
Let $\sum_k a_k$ converges so that $a_k^2 \to 0$  and $a_{k+1} a_{k-1} \to 0$ as $k \to +\infty$.
\end{lemma}
\begin{proof}
    The proof follows directly from the fact that if
    \begin{equation*}
        \sum_{k=0}^{+\infty} a_k = L
    \end{equation*}
    then 
    \begin{equation*}
        \lim_{k\varepsilon +\infty}a_k = 0,
    \end{equation*}
but then
    \begin{equation*}
        \lim_{k\varepsilon +\infty}a^2_k = 0,
    \end{equation*}
and 

    \begin{equation*}
        \lim_{k\varepsilon +\infty}a_{k+1}a_{k-1} = 0,
    \end{equation*}

\end{proof}
We are now ready to this preliminary result on the convergence of $\mathcal{L}$-serie.

\begin{theorem}
\label{theorem:series_without_log_concavity_assumption}
Let \((a_k)\) be a sequence of real numbers with the following properties:
\begin{itemize}
  \item \(\sum_k a_k\) converges absolutely (i.e., \(\sum_k |a_k| < \infty\)).
  \item \((a_k)\) is eventually monotonic (i.e., there exists \(N \in \mathbb{N}\) such that \((a_k)\) is monotonic for all \(k \geq N\)).
\end{itemize}
Then, the series \(\sum \mathcal{L}(a_k)\) converges.
\end{theorem}
\begin{proof}
If \(\sum_k a_k\) converges absolutely then we have we have \(a_k \to 0\) as \(k \to \infty\) and,  by the monotonicity of \((a_k)\) for \(k \geq N\), the terms \(a_k\) decay monotonically to 0.  
If \((a_k)\) is eventually monotonic then there exist and index $N$ such that 
\(a_{k-1} \geq a_k \geq a_{k+1}\) (if decreasing) or \(a_{k-1} \leq a_k \leq a_{k+1}\) (if increasing) and in both cases \(a_{k-1}a_{k+1} \leq a_k^2\), so \(\mathcal{L}(a_k) \geq 0\) for \(k \geq N\).  Since \(\mathcal{L}(a_k) \leq a_k^2\) for \(k \geq N\), and \(\sum_k a_k^2 < \infty\) (because \(\sum_k |a_k| < \infty\)), it follows that \(\sum \mathcal{L}(a_k)\) converges by the comparison test.
\end{proof}

We can state a stronger result if we introduce the hypothesis that $a_k$ is log-concave in particular:
\begin{theorem}
Let \((a_k)\) be a log-concave sequence. 
If \(\sum_k a_k\) converges, then:
\begin{itemize}
  \item The sequence \((a_k)\) decays exponentially, i.e., \(a_k \leq C r^k\) for some \(C > 0\) and \(0 < r < 1\).
  \item The series \(\sum \mathcal{L}(a_k)\) converges.
\end{itemize}
\end{theorem}

\begin{proof}
Let us start by proving under the assumptions that the $(a_k)$ sequence decays exponentially. Since \((a_k)\) is log-concave and \(\sum_k a_k\) converges, \(a_k \to 0\) as \(k \to \infty\) by proposition \ref{prop:convergence_under_log_concavity}, the log-concavity condition \(a_k^2 \geq a_{k-1}a_{k+1}\) implies that the sequence \((a_k)\) decays at least exponentially. Specifically, there exists \(C > 0\) and \(0 < r < 1\) such that \(a_k \leq C r^k\). Since $(a_k)$ decays exponentially both both \(a_k^2\) and \(a_{k-1}a_{k+1}\) decay exponentially and since \(\mathcal{L}(a_k) = a_k^2 - a_{k-1}a_{k+1} \geq 0\) because $(a_k)$ is log-concave we have   
\begin{equation*}
    \mathcal{L}(a_k) \leq a_k^2 \leq C^2 r^{2k}.
\end{equation*}
and so when we consider the associated series we have the following inequalities
\begin{equation*}
0 \leq  \sum_{k=0}^{+\infty}  \mathcal{L}(a_k) \leq \sum_{k=0}^{+\infty} a_k^2 \leq \sum_{k=0}^{+\infty} C^2 r^{2k}.
\end{equation*}
The series \(\sum_{k=0}^\infty C^2 r^{2k}\) is a convergent geometric series (since \(0 < r < 1\)). So by the comparison test, \(\sum \mathcal{L}(a_k)\) converges.   
\end{proof}
We can move further by obtaining a general result on the sequence of  {\it $i$-fold log\/} log operators  $\mathcal{L}^i(a_k)$. Let us start without the assumption of log-concavity, so only considering the log-concave operator properties.

\begin{theorem}
Let \((a_k)\) be a sequence of real numbers such that \(\sum_k a_k\) converges absolutely. Then, for all \(i \in \mathbb{N}\), the series \(\sum \mathcal{L}^i(a_k)\) converges.
\end{theorem}
\begin{proof}
With some attention, the proof is a checklist of the previous results coupled with the induction verification.
Let us start with the base case.
\newline 
(\(i = 1\)):
If \(\sum_k a_k\) converges, then \(a_k \to 0\) as \(k \to \infty\). By the log-concave operator definition, \(\mathcal{L}(a_k) = a_k^2 - a_{k-1}a_{k+1}\) we have that  since \(a_k \to 0\) as \(k \to +\infty \), both \(a_k^2\) and \(a_{k-1}a_{k+1}\) converge to 0, so \(\mathcal{L}(a_k) \to 0\). Moreover, \(\sum_k a_k^2\) converges (since \(\sum_k |a_k|\) converges), and \(\sum_k a_{k-1}a_{k+1}\) also converges (by the Cauchy-Schwarz inequality or comparison test). Thus, \(\sum \mathcal{L}(a_k)\) converges.
\newline
Inductive Step (\(i > 1\)) :
Assume that \(\sum \mathcal{L}^i(a_k)\) converges for some \(i \in \mathbb{N}\).
Since \(\mathcal{L}^i(a_k) \to 0\) as \(k \to \infty\), the sequence \(\mathcal{L}^i(a_k)\) satisfies the conditions for the base case.
Applying \(\mathcal{L}\) to \(\mathcal{L}^i(a_k)\), we have:
       \[
       \mathcal{L}^{i+1}(a_k) = \mathcal{L}(\mathcal{L}^i(a_k)) = (\mathcal{L}^i(a_k))^2 - \mathcal{L}^i(a_{k-1}) \mathcal{L}^i(a_{k+1}).
       \]
 By the inductive hypothesis, \(\sum \mathcal{L}^i(a_k)\) converges, so \(\sum (\mathcal{L}^i(a_k))^2\) and \(\sum \mathcal{L}^i(a_{k-1}) \mathcal{L}^i(a_{k+1})\) also converge.
 \newline
 Thus, \(\sum \mathcal{L}^{i+1}(a_k)\) converges. By induction, \(\sum \mathcal{L}^i(a_k)\) converges for all \(i \in \mathbb{N}\).
\end{proof}
We are ready to introduce also now a general result for $\mathcal{L}$-series by adding the log-concave hypothesis. 

\begin{theorem}
\label{th:-log-concave-series}
Let \((a_k)\) be a log-concave sequence. If \(\sum_{k=0}^\infty a_k\) converges, then the following results hold:
\begin{enumerate}
  \item The sequence \((a_k)\) decays exponentially.
  \item For all \(i \in \mathbb{N}\), the series \(\sum_{k=0}^\infty \mathcal{L}^i(a_k)\) converges.
  \item For all \(i \in \mathbb{N}\), the sequence \(\mathcal{L}^i(a_k)\) decays exponentially.
\end{enumerate}
\end{theorem}
\begin{proof}
    We break down the proof in various steps
    \newline
    \begin{enumerate}
        \item{\textbf{Exponentially decays of $(a_k)$}} Since \((a_k)\) is log-concave and \(\sum_{k=0}^\infty a_k\) converges, \(a_k \to 0\) as \(k \to \infty\). The log-concavity condition \(a_k^2 \geq a_{k-1}a_{k+1}\) implies that the sequence \((a_k)\) is unimodal or monotonic. Moreover, as we have already seen, for all \(k \geq 1\), \(\frac{a_{k+1}}{a_k} \leq \frac{a_k}{a_{k-1}}\), so the ratios \(\frac{a_{k+1}}{a_k}\) form a non-increasing sequence. Since \(a_k \to 0\), there exists \(N \in \mathbb{N}\) such that for all \(k \geq N\), \(\frac{a_{k+1}}{a_k} \leq r\) for some \(0 < r < 1\), iterating this inequality for all \(k \geq N\) we obtain \(a_k \leq a_N r^{k-N}\). Let 
        \begin{equation*}
         C = \frac{\max\{a_0, a_1, \dots, a_N\}}{r^N}
        \end{equation*}
        then for all for all \(k \geq 0\), \(a_k \leq C r^k\) so this prove that $(a_k)$ decays exponentially.   
    \newline
    \item{\textbf{Convergence of \(\sum \mathcal{L}^i(a_k)\)}}
    \newline
    We proceed by induction on \(i\).
    \newline
    \(i = 1\) (base case). By the previous step we know that under the assumptions and from the fact that the exponential decay of \((a_k)\), \(a_k \leq C r^k\), so \(a_k^2 \leq C^2 r^{2k}\) and \(a_{k-1}a_{k+1} \leq C^2 r^{2k}\). Thus, \(\mathcal{L}(a_k) \leq a_k^2 \leq C^2 r^{2k}\); since \(\sum_{k=0}^\infty r^{2k}\) converges, the comparison test implies that \(\sum_{k=0}^\infty \mathcal{L}(a_k)\) converges for $i =1$.
    \(i > 1\) (induction step). Let us assume that ssume that \(\sum_{k=0}^\infty \mathcal{L}^i(a_k)\) converges for some \(i \in \mathbb{N} , i > 1\). Since \(\mathcal{L}^i(a_k) \geq 0\) and \(\mathcal{L}^i(a_k) \to 0\) as \(k \to \infty\), the sequence \(\mathcal{L}^i(a_k)\) satisfies the conditions for the base case. Applying \(\mathcal{L}\) to \(\mathcal{L}^i(a_k)\), we have:
    \begin{equation*}
        \mathcal{L}^{i+1}(a_k) = \mathcal{L}(\mathcal{L}^i(a_k)) = (\mathcal{L}^i(a_k))^2 - \mathcal{L}^i(a_{k-1}) \mathcal{L}^i(a_{k+1}).
    \end{equation*}
    By the inductive hypothesis, \(\sum_{k=0}^\infty \mathcal{L}^i(a_k)\) converges, so \(\sum_{k=0}^\infty (\mathcal{L}^i(a_k))^2\) and \(\sum_{k=0}^\infty \mathcal{L}^i(a_{k-1}) \mathcal{L}^i(a_{k+1})\) also converge. Thus, \(\sum_{k=0}^\infty \mathcal{L}^{i+1}(a_k)\) converges, and by induction the thesis is proved.
    \newline
    \item{\textbf{Exponential Decay of \(\mathcal{L}^i(a_k)\)}}
    \newline
    As we did previously we prove the statement by induction on \(i\).
    \newline
    \(i = 1\) (base case). We have that \(a_k \leq C r^k\) because of the exponential decay property so 
    \(\mathcal{L}(a_k) \leq a_k^2 \leq C^2 r^{2k}\). Hence, \(\mathcal{L}(a_k)\) decays exponentially with \(C_1 = C^2\) and \(r_1 = r^2\).
    \newline
    \(i > 1\) (inductive step). Assume that \(\mathcal{L}^i(a_k) \leq C_i r_i^k\) for some \(C_i > 0\) and \(0 < r_i < 1\), by applying \(\mathcal{L}\) to \(\mathcal{L}^i(a_k)\), we have
    \begin{equation*}
        \mathcal{L}^{i+1}(a_k) = (\mathcal{L}^i(a_k))^2 - \mathcal{L}^i(a_{k-1}) \mathcal{L}^i(a_{k+1}).            
    \end{equation*}
    By the inductive hypothesis, \(\mathcal{L}^i(a_k) \leq C_i r_i^k\), so:
    \begin{equation*}
            (\mathcal{L}^i(a_k))^2 \leq C_i^2 r_i^{2k}, \quad \text{and} \quad \mathcal{L}^i(a_{k-1}) \mathcal{L}^i(a_{k+1}) \leq C_i^2 r_i^{2k}.
    \end{equation*}
    Thus, \(\mathcal{L}^{i+1}(a_k) \leq C_i^2 r_i^{2k}\). Hence, \(\mathcal{L}^{i+1}(a_k)\) decays exponentially with \(C_{i+1} = C_i^2\) and \(r_{i+1} = r_i^2\). By induction, \(\mathcal{L}^i(a_k)\) decays exponentially for all \(i \in \mathbb{N}\).
    \end{enumerate}

\end{proof}
\section*{Conclusion}

In this work, we have explored the properties of sequences and series under the action of the log-concave operator \(\mathcal{L}\), defined as:
\begin{equation*}
\mathcal{L}(a_k) = a_k^2 - a_{k-1}a_{k+1}.
\end{equation*}
Our primary focus has been on understanding the relationship between the convergence of a sequence \((a_k)\) and the convergence of the sequences and series derived from the iterative application of \(\mathcal{L}\) . Our key findings were the following implications.
Applying the \(\mathcal{L}\)-operator to a generic sequence $(a_k)$ than converges generate a sequence that converges. If $(a_k)$ is log-concave and the  associated series converges then also the \(\mathcal{L}\)-series converges. We have also analyzed the series associated to the iterated operator \(\mathcal{L}^i(a_k)\). If \((a_k)\) converges, then for all \(i \in \mathbb{N}\), the sequence \(\mathcal{L}^i(a_k)\) converges to $0$. Moreover, if \((a_k)\) is log-concave and \(\sum_k a_k\) converges, then for all \(i \in \mathbb{N}\), the series \(\sum \mathcal{L}^i(a_k)\) converges. We also established some results on exponential decays associated with the \(\mathcal{L}\) operator. In particular, we have that if \((a_k)\) is log-concave and \(\sum_k a_k\) converges, then \((a_k)\) decays exponentially. The iterated operator \(\mathcal{L}^i(a_k)\) also decays exponentially for all \(i \in \mathbb{N}\). By far, we have seen that log-concave sequences exhibit strong regularity properties, including exponential decay and convergence of associated series.
The results are useful in applications involving log-concave sequences, such as combinatorics, probability, and optimization.
Armed with this results some open question arise naturally. Only to cite we have seen some sufficent condition but  are the necessary conditions for the convergence of \(\sum \mathcal{L}(a_k)\) to imply the convergence of \(\sum_k a_k\)? We have established that there is a convergence in some cases but an explicit bounds on the rate of convergence of \(\mathcal{L}^i(a_k)\) could be interesting to find.  
As we have seen the log-concave condition is a strong one but there could be different results on to sequences that are not log-concave but satisfy weaker conditions. The study of log-concave sequences and their associated operators is a rich and fruitful area of research. The results presented in this work provide a solid foundation for further exploration, both theoretical and applied. By leveraging the structural properties of log-concave sequences, we have established strong convergence criteria for sequences and series derived from the log-concave operator \(\mathcal{L}\). These results not only deepen our understanding of log-concave sequences but also open up new avenues for research in related fields.

\bibliography{sn-bibliography}

\end{document}